\documentclass[11pt]{amsart}
\usepackage{amsfonts}

\usepackage[all,cmtip]{xy}
\usepackage{amsmath}
\usepackage{amsthm}
\usepackage{amssymb}
\usepackage{enumitem}
\newtheorem{thm}{Theorem}[section]

\newtheorem{prop}[thm]{Proposition}
\newtheorem*{definition*}         {Definition}
\newtheorem{lemma}[thm]{Lemma}

\newtheorem{cor}[thm]{Corollary}

\newtheorem*{remark}{Remark}
\theoremstyle{remark}

\newcommand*{\id}{\textrm{id}}

\newcommand*{\e}{\epsilon}
\newcommand*{\Q}{\mathbb{Q}}
\newcommand*{\E}{\mathcal{E}}
\newcommand*{\Hh}{\mathbb{H}}

\newcommand*{\Z}{\mathbb{Z}}
\newcommand*{\N}{\mathbb{N}}

\newcommand*{\R}{\mathbb{R}}

\newcommand*{\C}{\mathbb{C}}

\newcommand*{\PP}{\mathbb{P}}
\newcommand*{\FF}{\mathcal{F}}
\newcommand*{\D}{\mathbb{D}}

\DeclareMathOperator{\spec}{Spec}

\newcommand*{\ra}{\rightarrow}
\newcommand*{\ol}{\overline}
\def\SL{{\rm SL}}
\def\PPSL{\rm PSL}

\def\nd{{\rm nd}}
\def\n{{\rm n}}

\def\Hom{{\rm Hom}}

\def\stab{{\rm Stab}}
\def\rank{{\rm rank}}
\def\Hom{{\rm Hom}}

\def\im{{\rm Im}}
\def\codim{{\rm codim}}
\usepackage{amscd,amssymb,amsmath}
\usepackage{color}
\newcommand{\red}[1]{\textcolor{red}{#1}}
\newcommand{\blue}[1]{\textcolor{blue}{#1}}

\title{Ax-Schanuel for shimura varieties}
\author{Ngaiming Mok, Jonathan Pila,  and Jacob Tsimerman}
\begin{document}
\maketitle

\begin{abstract} 
We prove the Ax-Schanuel theorem for a general (pure) Shimura variety.
\end{abstract}

\bigskip

\centerline{CONTENTS}

\smallskip

\begin{itemize}

\item[1.] Introduction

\end{itemize}

\smallskip

\centerline{I. Basic Ax-Schanuel}

\smallskip

\begin{itemize}

\item[2.] Preliminaries

\item[3.] Some algebraicity results

\item[4.] Proof of Theorem 1.1.

\end{itemize}

\smallskip

\centerline{II. Ax-Schanuel with derivatives}

\smallskip

\begin{itemize}

\item[5.] Jet spaces

\item[6.] Differential equations

\item[7.] Schwarzians for Hermitian symmetric domains

\item[8.] Connection formula and automorphic functions 

\item[9.] Ax-Schanuel with derivatives and proof of Theorem 1.2

\item[10.] Proof of Theorem 9.1 and deduction of Theorem 1.3

\end{itemize}

\smallskip

\centerline{III. Ax-Schanuel in a differential field}

\smallskip

\begin{itemize}

\item[11.] Characterizing the uniformization map

\item[12.] Ax-Schanuel in a differential field and proof of Theorem 1.4

\smallskip

\item[]Acknowledgements

\item[]References

\end{itemize}

\section{Introduction}
Let $\Omega$ be a Hermitian bounded symmetric domain corresponding to a semisimple 
group $G$, and let $\Gamma\subset G(\Z)$ be a finite index subgroup. Then 
$X=\Gamma\backslash\Omega$ has the structure of a quasi-projective algebraic variety.
A variety $X$ arising in this way is called a {\it (connected, pure) Shimura variety\/}. 
We refer to \cite{DELIGNET, DELIGNEV} or \cite{MILNE}
for a detailed introduction to Shimura varieties.
A Shimura variety $X$ is endowed with a collection of {\it weakly special subvarieties\/}.
(There is a smaller collection of {\it special subvarieties\/}, which are 
precisely the weakly special subvarieties that contain a {\it special point\/};
these play no role in this paper). For a description of these see e.g. \cite{KUY}.

Let $q:\Omega\ra X$ be the natural projection map, and let $D\subset\Omega\times X$ 
be the graph of $q$.
Recall that $\Omega$ sits naturally as an open subset in its {\it compact dual\/} $\widehat{\Omega}$,
which has the structure of a projective variety. 
By an {\it algebraic subvariety\/} $W\subset \Omega\times X$ we mean a (complex analytically)
irreducible component of $\widehat{W}\cap (\Omega\times X)$ for some algebraic
subvariety $\widehat{W}\subset \widehat{\Omega}\times X$.
In the sequel, $\dim U$ denotes the complex dimension of a complex analytic set.
Though at some points we will refer implicitly to sets in real Euclidean spaces, any reference
to real dimension will be specifically noted.

Our basic result is the following. 

\begin{thm}\label{main} 
With notation as above, let $W\subset \Omega\times X$ be an algebraic subvariety. 
Let $U$ be an irreducible component of $W\cap D$ whose 
dimension  is larger than expected, that is, 
$$ \codim\, U < \codim\, W + \codim\, D, \leqno{(*)}$$ 
the codimensions being in $\Omega\times X$ or, equivalently,
$$
\dim W < \dim U + \dim X. \leqno{(**)}
$$
Then the projection of $U$ to $X$ is contained in a proper weakly special subvariety of $X$.
\end{thm}

If one takes $q: \Omega\rightarrow X$ to be the map 
$$
\exp: \C^n\rightarrow (\C^\times)^n,
$$
namely the Cartesian power of the complex exponential, then the statement is an
equivalent form of the Ax-Schanuel theorem of Ax \cite{AX}. 
In this form it is given a new proof in \cite{JTAX}.
Note however that $(\C^\times)^n$ is a ``mixed'' Shimura variety but not a ``pure'' one,
so this case is not formally covered by the above theorem.

One expects equality in $(*)$ and $(**)$ above,
on dimensional considerations, and such a component $U$ always has dimension
{\it at least\/} $\dim W-\dim X$ (see e.g. Lojasiewicz \cite{LOJASIEWICZ}, III.4.6). 
Thus the theorem asserts that all components of such intersections which are 
atypical in dimension are accounted for by weakly special subvarieties.

Since weakly special varieties are ``bi-algebraic'' in the sense of \cite{KUY}, 
they do indeed give rise to atypical intersections. For example, in the extreme 
case that $W=\Omega_1\times X_1$ where $X_1$ is a weakly special
subvariety of $X$ and $\Omega_1$ a connected component of its preimage in $\Omega$, 
we get $\dim W = \dim U + \dim X_1$.

As in earlier papers \cite{JTAX,KUY,PTAXJ}, the proof combines arguments
from complex geometry (Hwang-To), the geometry/group theory underlying Shimura varieties,
o-minimality, and monodromy (Deligne-Andr\'e).
The ingredients from o-minimality include the counting theorem of Pila-Wilkie, and results of
Peterzil-Starchenko giving powerful ``definable'' versions of the classical theorems of
Remmert-Stein and Chow.
Throughout the paper, we use `definable' to mean `definable in the o-minimal structure 
$\R_{\textrm{an,exp}}$'; see \cite{DM} and \cite{DMM}, where the o-minimality of 
$\R_{\textrm{an,exp}}$ is established, building on \cite{WILKIE}.

The crucial new ingredient in this paper is the observation of additional algebraicity 
properties, and for these we make further essential use of the above-mentioned results 
of Peterzil-Starchenko. However, there is also a purely complex analytic approach, 
which we allude to below.

We expect that Theorem 1.1, which is sometimes called the ``hyperbolic Ax-Schanuel 
conjecture''  \cite{DAWREN}, will find applications to the Zilber-Pink conjecture, 
where it can play a role analogous to that of the ``Ax-Lindemann theorem'', which it generalizes, 
in proving cases of the Andr\'e-Oort conjecture, see e.g. \cite{PTAS}.
One such application has been given by Daw-Ren \cite{DAWREN}. 
An application in a different direction is given
in \cite{ACZ}. A generalization to 
variations of Hodge structures is given by Bakker-Tsimerman \cite{BT}.

We will prove a strengthening of Theorem 1.1 involving the uniformizing function together with its
derivatives, along the same lines as the result of Pila-Tsimerman \cite{PTAS} for Cartesian powers of the $j$-function.
For this we first observe the following generalization of a result of Bertrand-Zudilin \cite{BEZU}
in the case of the Siegel modular varieties (but note that their result holds over any algebraically
closed subfield of $\C$; we do not have any control over the field of definition
of the differential equations).
The following result is established in \S9 as Corollary 9.3.
Let $N^+\subset G$ be the unipotent radical of an opposite parabolic subgroup of the complex parabolic subgroup $B\subset G$ defining the symmetric space $\widehat{\Omega} \ ( = G/B)$.

\begin{thm}
Let $z_1,\dots,z_n$ be an $N^+(\C)$-invariant
algebraic coordinate system on $\Omega$. 
Let $\{\phi_1,\dots,\phi_N\}$ be a $\C$-basis of modular functions.
Then the field generated by $\{\phi_i\}$ and their partial derivatives
with respect to the $z_j$ up to order $k\geq 2$ has transcendence 
degree over $\C$ equal to $\dim G$. Further, the transcendence degree is the same 
over $\C(z_1,\dots,z_n)$.
\end{thm}

To frame our result we need to study the
form of the differential equations satisfied by the uniformization map, 
for which we introduce and study, in \S7 and \S8, the 
{\it Schwarzian derivative\/} for a Hermitian symmetric domain.
Differential equations associated with covering maps are studied 
by Scanlon \cite{SCANLON}, who shows under quite general assumptions
that one gets algebraic differential equations. 
A key ingredient there, as here, is definability
and the results of Peterzil-Starchenko. However, our focus is on getting more specific
information (such as 1.2 above)  in the special case of Shimura varieties.
On this circle of differential ideas see also Buium \cite{BUIUM}.
For a description of the Schwarzian for ${\rm PSL}_m$ see \cite{YOSHIDA}.

Our Ax-Schanuel theorem for $q$ and its derivatives (Theorem 9.1) is most naturally stated in the
setting of {\it jet spaces\/}. These are introduced in \S5 and \S6. 
Here we give the following jet-space-free consequence of Theorem 9.1.

\begin{thm}\label{trdeg}

Let $V\subset \Omega$ be an irreducible complex analytic variety, not contained in a proper 
weakly special subvariety. Let  $\{z_i, i=1,\ldots, n\}$ be an algebraic 
coordinate system on $\Omega$. Let $\{\phi_j^{(\nu)}\}$ consist of a basis $\phi_1,\ldots, \phi_N$ of 
modular functions, all defined at at least one point of $V$,
together with their partial derivatives with respect to the $z_j$ 
up to order $k\ge 2$. Then
$$
{\rm tr.deg.}_{\C}\C\big(\{z_i\}, \{\phi_j^{(\nu)}\}\big) \geq \dim G + \dim V
$$	
where all functions are considered restricted to $V$.

\end{thm}

We also give, in \S12, a version of Ax-Schanuel in the setting of a differential field, 
and show that it in turn directly implies the jet version. This depends on the fact that 
all solutions of the relevant differential system
are $G$-translates of $q$, and this is due to the provenance of the system in properties of the
Schwarzian associated with $G$ and $\Omega$.

Let $(K, \mathcal{D}, C)$ be a differential field with   a finite set $\mathcal{D}$
of commuting derivations and constant field $C$.
We consider $K$-points $(z,x,y)$ of suitable varieties over $C$
and establish a differential algebraic condition under which such a point corresponds
to a locus $z$ in $\widehat{\Omega}$, whose dimension equals the {\it rank\/} 
of $z$, a corresponding locus $x$ in $X$ under 
some $G(\C)$-translate $Q$ of $q$, and the restrictions $y$ of suitable derivatives of 
$Q$ to the locus $x$.

The precise definition of such a {\it uniformized locus of rank $k$} in $K$
is given \S12,  after the differential algebraic condition is established in \S11. 
Under suitable
identifications, we can also speak of $x$ being {\it contained in a proper weakly special
subvariety\/} of $X$. With these notions, the differential version of Ax-Schanuel
may be stated as follows.

\begin{thm}[Differential Ax-Schanuel]
Let $G, q, (K, \mathcal{D}, C)$ be as above.
Let $(z, x, y)$ be a uniformized locus of rank $k$. Then
$$
{\rm tr. deg.}_CC(z, x, y)\ge {\rm rank}(z) +\dim G
$$
unless $x$ is contained in a proper weakly special subvariety.
\end{thm}

\bigbreak

\centerline{I. Basic Ax-Schanuel}

\bigskip

This part is devoted to the proof of Theorem 1.1.

\section{Preliminaries}

We gather some preliminary remarks, definitions, and results. 

\medbreak
\noindent
{\it 2.1. Shimura varieties}

\medskip

According to the definition given, a Shimura variety $X$ may not be smooth,
and the covering $q: \Omega\rightarrow X$ may be ramified, if $\Gamma$ contains
elliptic elements. For example, $j: \Hh\rightarrow\C$ is ramified at ${\rm SL}_2(\Z)i$
and ${\rm SL}_2(\Z)\rho$, where $\rho=\exp(2\pi i/3)$, even though in this case the quotient $\C$ is still smooth.

By passing to a finite index subgroup we may always assume that the uniformization
is unramified and the Shimura variety is smooth, and hence a complex manifold.
This does not affect the validity of 1.1.
Hence we may and do assume throughout that $X$ is smooth and that $q:\Omega\rightarrow X$ is 
unramified.

\medbreak
\noindent
{\it 2.2 Definability}
\medskip

The definition and basic results
on o-minimal structures over a real closed field may be found in \cite{PSICM}. 
As already mentioned,
`definable' will mean `definable in the o-minimal structure $\R_{\rm an, exp}$'.

Let $\mathcal {F}$ be the classical Siegel domain for the action of 
$\Gamma$ on $\Omega$. Then
the uniformization $q:\Omega\rightarrow X$ restricted to $\mathcal{F}$
is {\it definable\/}.
For a general Shimura variety this result is due to Klingler-Ullmo-Yafaev \cite{KUY},
generalizing results of Peterzil-Starchenko \cite{PSDMJ} for moduli spaces of abelian varieties.

We shall need the following results, which can be see as definable generalizations of GAGA-type theorems.

\begin{thm}\label{defremmertstein}[``Definable Remmert-Stein'', \cite{PSICM}, Theorem 65.3]
Let $M$ be a definable complex manifold and $E$ a definable complex analytic subset
of $M$. If $A$ is a definable, complex analytic subset of $M-E$ then
its closure $\overline{A}$ is a complex analytic subset of $M$. \ \qed
\end{thm}

The following is a slight generalization of a theorem stated by Peterzil-Starchenko
\cite{PSICM}, Theorem 4.5,
which may be proved by combining their statement with 
``Definable Remmert-Stein'' above.
This strengthening has also been observed by Scanlon \cite{SCANLON}, 
Theorem 2.11, and, in a slightly less general form,  in \cite{PTAS}.

\begin{thm}\label{defchow}[``Definable Chow'']
Let $Y$ be a quasiprojective algebraic variety, and let $A\subset Y$ be definable,  
complex analytic, and closed in $Y$. Then $A$ is algebraic.
\end{thm}	

\begin{proof} We follow the proof in \cite{PTAS}.
By taking an affine open set in $Y$, it suffices to consider the case 
where  $Y$ is an affine subset of projective space. Then 
$Y$ is a definable, complex analytic subset of $M\backslash E$ where $M$ is a 
projective variety and $E$ is a closed algebraic subset of $M$. 
Then, by ``Definable Remmert-Stein'', above, 
the closure of $A$ in $M$ is a definable, complex analytic subset of $M$, 
hence complex analytic in the ambient projective space.
Thus $A$ must be algebraic by Chow's theorem,
or by the Peterzil-Starchenko version \cite{PSICM}, Theorem 4.5.
\end{proof}

We say that a set is ``constructible complex analytic" if it is in the boolean algebra generated by closed, complex analytic varieties.

\begin{cor}\label{consdefchow}
Let $Y$ be a quasiprojective algebraic variety, and let $A\subset Y$ be definable,  
constructible complex analytic. Then $A$ is constructible algebraic.
\end{cor}

\begin{proof} Follows from Theorem \ref{defchow} by induction on dimension. 
\end{proof}


\section{Some algebraicity results}

\medskip

We have the uniformization $q:\Omega\rightarrow X$
in which $X$ is a quasi-projective variety, the map $q$ is complex analytic and surjective.
It is further invariant under the action of some discrete group $\Gamma$ of
holomorphic automorphisms of $\Omega$, and as noted above the restriction of
$q$ to a suitable fundamental domain $\FF$ for this action is definable.

Suppose that $V\subset X$ is a relatively closed algebraic subvariety. 
Then $q^{-1}(V)\subset\Omega$ is a closed complex analytic set which 
is $\Gamma$-invariant, and definable on a fundamental domain $\FF$.
The same statement holds for the uniformization 
$$q\times{\rm id}: \Omega\times X\rightarrow X\times X$$
and $V\subset X\times X$,
which is invariant under $\Gamma\times\{{\rm id}\}$, where now 
$q\times{\rm id}$ is definable on $\mathcal{F}\times X$.
We observe that the converse holds.

\begin{thm}\label{alginXxX}
Let $A\subset\Omega\times X$ be a closed, complex analytic set which is 
$\Gamma\times\{{\rm id}\}$-invariant, and such that $A\cap\FF\times X$ is definable.
Then $(q\times{\rm id})(A)\subset X\times X$ is a closed algebraic subset.
\end{thm}

\begin{proof}
The image $(q\times{\rm id})(A)$ is closed and complex analytic in $X\times X$. Since
$(q\times{\rm id})(A) = (q\times{\rm id})(A\cap \FF\times X)$ it is also definable, 
and so it is algebraic by ``Definable Chow'' (\ref{defchow}).
\end{proof}

\begin{cor}
Let $A\subset\Omega\times X$ be a closed, constructible complex analytic set which is 
$\Gamma\times\{{\rm id}\}$-invariant, and such that $A\cap\FF\times X$ is definable.
Then $(q\times{\rm id})(A)\subset X\times X$ is a constructible algebraic subset.
\end{cor}

\begin{proof} Follows from Corollary \ref{consdefchow} as above.
\end{proof}

\subsection{Descending Hilbert scheme loci}\label{hilbert}

Fix a smooth, projective compactification $\widehat{X}$ of $X$. Now we fix some algebraic subvariety $W\subset \Omega\times X$, 
with  $\widehat{W}\subset\widehat{\Omega}\times \widehat{X}$
its Zariski closure, and $U$ an irreducible component of $W\cap D$.
We make no assumptions here on the dimension of $U$.
By the {\it Hilbert polynomial\/} $P=P_W(\nu)$ of $W$
we mean the Hilbert polynomial of $\widehat{W}$.

Let $M$ be the Hilbert 
scheme of all subvarieties of $\widehat{\Omega}\times \widehat{X}$ with  Hilbert polynomial $P$.
Then $M$ also has the structure of an algebraic variety. 
Corresponding to $y\in M$ we have the subvariety $W_y\subset \Omega\times X$,
and we have the incidence variety (universal family)
%
%
\[
B=\{(z,x,y)\in \Omega\times X\times M: (z,x)\in W_y\},
\]
and the family of the intersections of its fibres over $M$ with $D$, namely
\[
A=\{(z,x,y)\in \Omega\times X\times M: (z,x)\in W_y\cap D\}.
\]
Then $A$ is a closed complex analytic subset of $\Omega\times X\times M$. 
It has natural projection $\theta: A\rightarrow M$, with $(z,x,y)\mapsto y$. 
Then, for each natural number $k$, the set
\[
A(k)=\{(z,x,y)\in \Omega\times X\times M: \dim_{(z,x)}\theta^{-1}\theta(z,x,y)\ge k\},
\]
the dimension being the dimension at $(z,x)$ of the fibre of the projection in $A$,
is closed and complex analytic see e.g. the proof of \cite{PSNEWTON}, Lemma 8.2,
and references there.

Now we have the projection $\psi: \Omega\times X\times M\rightarrow\Omega\times X$,
and consider$$Z=Z(k)=\psi(A(k)).$$
Then as $M$ is compact, $\psi$ is proper and so $Z$ is closed in $\Omega\times X$. 
Note that $Z$ is $\Gamma$-invariant
and $Z\cap\FF\times X$ is definable. 
%

\begin{lemma}\label{algT}
Let $T=(q\times{\rm id})(Z)$. Then $T\subset X\times X$ is closed 
and algebraic.
\end{lemma}

\begin{proof}
Since $Z$ is $\Gamma$-invariant and $Z\cap\FF\times X$ is definable, this follows as in Theorem \ref{alginXxX}.
\end{proof}

\begin{remark}

One may also prove Lemma \ref{algT} by more geometric methods along the
lines of the argument in \cite{MOK}, which uses the method of compactification of complete
K\"ahler manifolds of finite volume of \cite{MOKZHONG} based on
$L^2$-estimates of $\overline{\partial}$.


\end{remark}

\medskip

\bigbreak

\section{Proof of Theorem 1.1}

\begin{proof}
We argue by induction, in the first instance (upward) on $\dim\Omega$.
For a given $\dim\Omega$, we argue (upward) on $\dim W-\dim U$.
Finally, we argue by induction (downward) on $\dim U$.

We  carry out the constructions of \S3.1 with $k=\dim U$ and keep the notation there. 
Let $A(k)'\subset A(k)$  be the irreducible component which contains 
$U\times [W]$, where $[W]$ is the moduli point of $W$ in $M$.
Let $Z'=\psi(A(k)')\subset Z$ be the
corresponding irreducible component of $Z$, and let $V=(q\times{\rm id})(Z')$ be the 
irreducible component of $T$, which is therefore algebraic by Lemma \ref{algT}.
Now, by assumption $V$ contains $(q\times {\rm id})(U)$, and so it is not contained in any proper 
weakly special of the diagonal $\Delta_X$, and thus has Zariski-dense monodromy
by Andr\'e-Deligne \cite{ANDRE}.

Consider the family  $F_0$ of algebraic varieties corresponding to $A(k)'$. 
Let $\Gamma_0\subset\Gamma$ be the subgroup of elements $\gamma$ 
such that every member of $F_0$ is invariant by $\gamma$. 
For any $\mu\in\Gamma$ define $E_\mu\subset F_0$ to be the subset corresponding
to algebraic subvarieties invariant under $\mu$. Then, for $\mu\in\Gamma-\Gamma_0$, 
$E_\mu\subsetneq F_0$ is an algebraic subvariety.
Hence, a very general\footnote{in the sense of being the complement of countably 
many proper subvarieties} element $W'$ of $F_0$ is  invariant by exactly the subset $\Gamma_0$ of 
$\Gamma$. Let $\Theta$ be the connected component of the 
Zariski closure of $\Gamma_0$ in $G(\R)$. 

\medskip

\begin{lemma} \label{normalsub}

$\Theta$ is a normal subgroup of $G(\R)$. 

\end{lemma}

\begin{proof}

Note that there is an action of $\Gamma$ on $A(k)$ given by 
$\gamma\cdot (z,x,[W])=(\gamma z,x,[\gamma W])$, and the map $A(k)\ra Z$ is equivariant 
with respect to this action. Since $A(k)\ra Z$ is proper and the action of $\Gamma$ is discrete, 
it follows that $\Gamma\backslash A(k)\rightarrow \Gamma\backslash Z\cong T$ 
is a proper map of analytic varieties.

Let $\Gamma\backslash A(k)'$ denote the image of $A(k)'$ in $\Gamma\backslash A(k)$. 
Then $\phi:\Gamma\backslash A(k)'\rightarrow V$ is a proper map of analytic varieties, 
and thus the fibers of $\phi$ have finitely many components. Let $\Gamma_1$ be the image of 
$\pi_1(\Gamma\backslash A(k)')\ra\pi_1(V)\ra\Gamma$. 
Then $\pi_1(A(k)')$ has finite index image in the monodromy 
group of $V$, and thus $\Gamma_1$ is Zariski-dense in $G(\R)$ by 
Andr\'e-Deligne \cite{ANDRE}.

It is clear that $F_0$ is invariant under $\Gamma_1$. 

Now, letting ${\rm stab}(W')$ denote the stabilizer of $W'$, we have 
${\rm stab}(\gamma\cdot W')=\gamma\cdot{\rm stab}(W')\cdot\gamma^{-1}$. 
It follows that $\Gamma_0$, and hence also
$\Theta$ is invariant under conjugation by $\Gamma_1$, and hence under its Zariski closure, 
which is all of $G(\R)$. This completes the proof.
\end{proof}
%
%

\medskip

\begin{lemma} $\Theta$ is the identity subgroup.

\end{lemma}

\begin{proof}

We argue by contradiction. 
Without loss of generality we may assume that $W$ is a very general member of $F_0$, 
and hence is invariant by exactly $\Theta$.
Since $\Theta$ is a $\Q$-group by construction, 
it follows that $G$ is isogenous to $\Theta\times\Theta'$ and we have a 
splitting $\Omega=\Omega_{\Theta}\times\Omega_{\Theta'}$ of Hermitian symmetric domains. 
Replacing $\Gamma$ by a finite index subgroup we 
also have a splitting $X\cong X_{\Theta}\times X_{\Theta'}$.

Now, as $W$ is invariant under $\Theta$ it is of the form $\Omega_{\Theta}\times W_1$ 
where $W_1\subset \Omega_{\Theta'}\times X_{\Theta'}\times X_{\Theta}$. Moreover, $D$
splits as $D_{\Theta}\times D_{\Theta'}$. Let $U_1$ be the projection of $U$ to $W_1$. 
Since the map from $D_{\Theta}$ to $X_{\Theta}$ has discrete pre-images, 
it follows that $\dim U = \dim U_1$. 

Now Let $W'$ be the projection of $W_1$ to $\Omega_{\Theta'}\times X_{\Theta'}$. 
Then letting $U'$ be a component of $W'\cap D_{\Theta'}$ we easily 
see that $U'$ is the projection of $U$ to $\Omega_{\Theta'}\times X_{\Theta'}$. Now let $W''$ be the 
Zariski closure of $U'$. It follows by induction that $$\dim U'+\dim X_{\Theta'}\leq \dim W''.$$

Now for the projection map $\psi: W_1\ra W'$, the generic fiber dimension of $W''$ 
is the same as the generic fiber dimension over $U'$, and thus 
$$\dim U_1 + \dim X_{\Theta'}\leq \dim \psi^{-1}(W'')\leq \dim W_1,$$ 
from which it follows that  
%
%
$$\dim U +\dim X \leq \dim W$$ as desired.
\end{proof}

It follows that $W$ is not invariant by any infinite subgroup of $\Gamma$.  
The following lemma thus reaches a contradiction, and completes the proof. 

\begin{lemma}\label{infinitesub}

$W$ is invariant by an infinite subgroup of $\Gamma$.

\end{lemma}

\begin{proof}

As before, let $\FF$ be a fundamental domain for $\Gamma$
on which the map $q$ is definable, and consider the definable set 
$$
I=\{\gamma\in G(\R) \mid 
\dim_\R \big((\gamma\cdot W)\cap D\cap(\FF\times X)\big)=\dim_\R U\}.
$$ 
Clearly, $I$ contains $\gamma\in\Gamma$ 
whenever $U$ intersects $\gamma\FF\times X$.  

We claim that the $\dim_\R U$ -dimensional volume of $\big((\gamma\cdot W)\cap D\cap(\FF\times X)\big)$ is uniformly bounded over $\gamma\in I$.
To see this, we proceed as in the work of Klingler-Ullmo-Yafaev \cite[Lemma 5.8]{KUY}. They show that using Siegel coordinates one can cover $\FF$ by a finite union of sets 
$\Theta$ which embed into a product of 1-dimensional sets $\prod _{i=1}^m J_i$, and that there are (1,1)-volume forms $\omega_i$ on $J_i$ such that $\sum_i \omega_i$ dominates the
K\"ahler form of an invariant hyperbolic metric, and $\int_{J_i} \omega_i < \infty$. For a subset $I\subset\{1,\dots,m\}$ containing $d=\dim_{\mathbb R} U$ elements, let $J_I=\prod _{i\in I} J_i$. Thus, it suffices to show that the projections of 
$\big((\gamma\cdot W)\cap D\cap(\FF\times X)\big) \cap \Theta$ onto $J_I$ have uniformly bounded finite fibers. However, this is an immediate consequence of the definability, which establishes the claim.

By the work of Hwang--To \cite[Theorem 2]{HT}, the volume of $\gamma\cdot W\cap X\cap B(R)$ grows exponentially with $R$, where $B(R)$ is a hyperbolic ball of radius $R$. Thus, as in \cite{PTAXJ}, it follows
that $I$ contains a polynomial number of integer matrices\footnote{Ordered by height, 
there are at least $T^{\delta}$ integer points of height at most $T$ for some fixed 
$\delta>0$ and arbitrarily large $T$.}. It 
follows by the Pila-Wilkie theorem \cite{PW} that $I$ contains irreducible 
real algebraic curves $C$ containing arbitrarily many integer points, in particular, 
containing at least 2 integer points. 

Write $W_c := c\cdot W$. If $W_c$ is constant in $c$, then $W$ is stable under $C\cdot C^{-1}$. Since $C$ 
contains at least 2 integer points, it follows that $W$ is stabilized by
a non-identity integer point, completing the proof that $W$ is invariant under an infinite group (since $\Gamma$ is torsion free). 
So we assume that $W_c$ varies with $c\in C$. Note that since $C$ contains an
 integer point that $(q\times {\rm id})(W_c\cap D)$ is not contained in a weakly special subvariety 
 for at least one $c\in C$, and thus for all but a countable subset of $C$ 
 (since there are only countably many families of weakly special subvarieties).

We now have 2 cases to consider. First, suppose that $U\subset W_c$ for $c\in C$. 
Then we may replace $W$ by $W_c\cap W_{c'}$ for generic $c,c'\in C$ 
and lower $\dim W$, contradicting our
induction hypothesis on $\dim W-\dim U$. 

On the other hand, if it is not true that $U\subset W_c$ for $c\in C$ then $W_c\cap D$ 
varies with $C$, and so we may set $W'$ to be the Zariski closure of $C\cdot W$. 
This increases the dimension of $W$ by $1$, but then $\dim W'\cap D = \dim U + 1$ as well, 
and thus we again contradict our induction hypothesis, this time on $\dim U$. 
This completes the proof. \end{proof}

%
%
%

This contradiction completes the proof of Theorem 1.1.
\end{proof}

\bigbreak

\centerline{II. Ax-Schanuel with derivatives\/}

\bigskip

In this part we establish versions of Ax-Schanuel for $q$ together with its derivatives.
The result is formulated in the setting of jet spaces.

\section{Jet Spaces}

\subsection{Definition}

Let $X$ be a smooth complex algebraic variety and $k,g\geq 1$ be positive integers. 
Set
$$
\D^g_k=\spec\C[\e_1,\dots,\e_g]/M^{k+1}
$$ 
where $M$ is the ideal  $(\e_1, \e_2,\dots, \e_g)$. 
We define the {\it jet space of order $k$\/} to be
$$
J^g_kX:=\Hom(\D^g_k,X).
$$ 
Note that $J^g_0X=X$ and $J^1_1X$ is the tangent bundle of $X$. 
It is evident that $J^g_k$ is a functor and also that there are natural projection maps 
$J^g_k\rightarrow J^g_r$ whenever $k>r$. As a matter of notation, we write simply $J_kX$ 
to denote $J_k^{\dim X} X$. 

\subsection{ Maps between Jet Spaces}

For $a,b>0$ there is a natural map 
$$\pi_{a,b}^\#:\C[\e_1,\dots,\e_g]/M^{a+b} \ra \C[t_1,\dots,t_g]/M^a \otimes \C[s_1,\dots,s_g]/M^b$$ 
defined by $\phi(\e_i)=s_i+t_i$. This induces a map
$\phi:\D^g_a\times \D^g_b\ra \D^g_{a+b}$, which is just the truncation of the addition map. 
Now there are natural identifications
$$
J_aJ_bX\cong \Hom(\D^g_a,\Hom(\D^g_b,X))\cong \Hom(\D^g_a\times \D^g_b,X)
$$ 
and therefore $\phi$ induces a natural map $$\pi_{a,b}:J_{a+b}X\ra J_aJ_bX.$$ 
Since $\pi_{a+b}^\#$ is injective as a map of rings, $\pi_{a,b}$ is injective on the level of points. 

Moreover, since $J_{a+b}$ is postcomposition and $\pi_{a,b}$ is precomposition, it is easy to see that they commute. In other words, for a map $f:X\ra Y$, we have 
\begin{equation}\label{J1}
\pi_{a,b}\circ J_{a+b}f = J_a(J_b f)\circ \pi_{a,b}.
\end{equation}

To see this, consider the following diagram:

$$\D^g_a\times \D^g_b\ra \D^g_{a+b} \ra X \ra Y.$$

\section{Differential equations}

\subsection{The jet space formulation}

Suppose  $\phi:\C^g\ra \C^g$ is a holomorphic function which satisfies an algebraic
differential equation of degree $m$ given by a set of polynomials in the derivatives of the 
components of $\phi$, which we may write as 
$\vec{F}(\frac{\partial^{|J|}\phi_j}{\partial z_J})_{|J|<m,j\leq g}$=0. 
We record this geometrically as follows: 
Consider the natural section $\id_m:\C^g\ra J_m\C^g$ given by 
$\id_m(a):\D_m\ra \C^g$, where the latter is given by 
$$
\id_m(a)^{\#}(z_1,\dots,z_g)=(a_1+\e_1,\dots,a_g+\e_g).
$$
Then, in the natural coordinates for $J_mX$, the partial derivatives of $\phi$ are the coordinates of
$J_m\phi\circ \id_m: \C^g\ra J_k\C^g$. 

We may record our differential equation in the following way: there is a subvariety $W\subset J_m\C^g$ 
and we insist that $J_m\phi\circ\id_m(\C^g)\subset W$. Of course, we may then replace the target 
$\C^g$ with any space $X$ and formulate a differential equation by picking a subvariety 
$W\subset J_mX$ without having to pick local coordinates on the image.

\subsection{Differentiating a differential equation}

If $f(z)$ satisfies $f'(z)=R(f(z))$ then it also satisfies $f''(z)=f'(z)R'(f(z))=R(f(z))R'(f(z))$. 
We explain now how to derive this relation geometrically. 

First, note that $J_a\id_b\circ\id_a(k)$ is the local map $s\mapsto(t\mapsto z+t+s)$.
Thus, we see that 
\begin{equation}\label{J2}
J_a\id_b\circ\id_a = \pi_{a,b}\circ\id_{a+b}.
\end{equation} 

Now suppose that $\phi:\C^g\ra X$ satisfies $\im J_b\phi\circ\id_b\subset W$. For any $a>0$ we have
\begin{align*}
\pi_{a,b}J_{a+b}\phi\circ \id_{a+b}&= J_a(J_b\phi)\circ \pi_{a,b}\circ\id_{a+b}\textrm{ by }\eqref{J1}\\
&= J_a(J_b\phi)\circ J_a\id_b\circ \id_a\textrm{ by }\eqref{J2}\\
&\subset J_aW\subset J_aJ_bX.
\end{align*}
Thus, we learn that $\im J_{a+b}\phi\circ \id_{a+b}$ is contained in $\pi_{a,b}^{-1}(J_a W)$. 

\subsubsection{Example}

Suppose that $\phi'(z)=R(\phi(z))$. Let $X=\C$, so $J_1X$ can be identified with $\mathbb{A}^2_{z,r}$ 
by the maps $t\mapsto z+rt$.  Let $W\subset J_1X$ be defined by the relation $r=R(z)$, so that
$\im J_1\phi\circ\id_1\subset W$.  Now $J_1J_1X$
can be defined by $(z,r,z_1,r_1)$ by 
$$
s\mapsto(t\mapsto z+rt+sz_1+str_1).
$$

Now an element of $J_1W$ has the form $s\mapsto (z+es,R(z)+eR'(z)s)$ and so it maps to 
$$
s\mapsto (t\mapsto z+es+tR(z)+teR'(z)s).
$$
Now the image $\pi_{1,1}J_2X$ in $J_1J_1X$ consists of those elements 
that are functions of $s+t$ in the ring $\C[s,t]/(s^2,t^2)$. Thus, we need the $s$ and $t$-coefficients 
to be the same, so we must have  $e=R(z)$, and we get the map 
$$s+t\ra z+R(z)(s+t)+\frac12 R(z)R'(z)(s+t)^2.$$ 

Note that this exactly records the relation $\phi''(z) = R(\phi(z))R'((\phi(z))$ as desired.

\section{Schwarzians for Hermitian symmetric spaces}

\subsection{Setup}

Let $\Omega$ be a Hermitian symmetric space of dimension $n$. We may write $\Omega$ as 
$G_{\R}/K$ for a semisimple group $G$ and maximal compact real subgroup $K$ inside it. Then
$\Omega$ sits naturally inside the flag manifold $\widehat{\Omega}:=G_{\C}/B$ for a complex 
parabolic subgroup $B$. Now the lie algebra of $G$ decomposes as 
$\frak{g}=\frak{n}^-\oplus \frak{k_{\C}}\oplus \frak{n}^+$giving $\frak{g}$ 
a Hodge structure of weight 2, where $\frak{k}$  is the 
lie algebra of $K$ and $\frak{b}=\frak{n}^-\oplus\frak{k_{\C}}$ is the lie algebra of $B$. Let $N^{-},N^{+}$ 
be the corresponding (abelian) unipotent groups. Picking a base-point $o\in X$, 
we may give coordinates on an open subset of $\widehat{\Omega}$ by identifying it with 
$N^{+}$ by $\nu\ra \nu\cdot o.$ Fixing an identification $\C^n\ra N^{+}$ of vector spaces 
once and for all, we get sections
$\id_m:N^+\ra J_mN^+$ which are compatible and invariant under $N^{+}$. 

We would like to characterize those functions $F:  \widehat{\Omega}\ra \widehat{\Omega}$ 
which look like $F(z)=g\cdot z$ for an element $g\in G_{\C}$. As such, we define the 
{\it Schwarzian differential equation\/} to be 
$$
W_m:=G\cdot\id_m(N^+)=G\cdot\id_m(o)\subset J_m\widehat{\Omega}.
$$

\begin{definition*}

We define $J^{\nd,r}_kY\subset J^r_kY$ to be all those infinitesimal maps which are surjective on tangent spaces. Note that this is only non-empty for $r\geq \dim Y$. 

\end{definition*}

\subsubsection{Example}

Consider $G=\PPSL_2, \Omega=\Hh, o=0\in \C$. Let $N^{-}$ be the lower 
triangular matrices and $N^+$ 
the upper triangular matrices. Then $J_3\C$ 
can naturally be given coordinates $(z,a,b,c)$ corresponding to the
map $t\ra z+at+\frac{bt^2}{2}+\frac{ct^3}{6}$. Since $W_3$ is $N^+$ invariant it must be cut out 
by a function of $a,b,c$ and it is sufficient to consider it at $z=0$, which is fixed by all
of the lower triangular matrices. Acting first by a diagonal element we transform 
$(0,a,b,c)$ to $(0,1,b/a,c/a)$ 
via $z\ra z/a$. Now acting by a lower triangular matrix $z\ra \frac{z}{\frac{b}{2a}z+1}$  we get
$$\frac{t+\frac{b}{2a}t^2+\frac{c}{6a}t^3}{1+\frac{b}{2a}t+\frac{b^2}{4a^2}t^2+\frac{bc}{12a^2}t^3} = 
t+(\frac{c}{6a}-\frac{b^2}{4a^2})t^3$$ 
which transforms $(0,1,b/a,c/a)$ to $(0,1,0,\frac{c}{6a}-\frac{b^2}{4a^2})$. 
Thus we recover the classical Schwarzian in this setting.
 
 \subsection{Fixing the order 2 infinitesimal neighborhood}

 We have the following lemma:
 
 \begin{lemma} \label{fix1}
 
 The subgroup of $B$ which fixes $\rm{id}_1(o)$ is $N^{-}$.
 
 \end{lemma}
 
 \begin{proof}
 
We are looking for elements which act trivially on the tangent space at $o$,
$T_{o}\widehat{\Omega}\cong \frak{g}/\frak{b}$.  
The action of $B$ on this space is induced the adjoint action. 
 
We first claim that $N^{-}$ acts trivially. To see this, note that the Lie bracket respects
the Hodge grading on $\frak{g}$, so $[\frak{n}^-,\frak{n}^-]=0$ and 
$[\frak{n}^-,\frak{k}]\subset \frak{n}^-$, 
from which the claim follows.
 
Since $B=K_{\C}N^+$,to finish the proof it suffices to show that no element of $K_{\C}$ acts trivially.
Conjugation by $K_{\C}$ preserves $\frak{n^+}$ so if an element $k\in K_{\C}$ acts trivially on 
$T_o\widehat{\Omega}$ it must also act trivially on $\frak{n}^+$ and thus commute with $N^+$. 
Since $N^+o$ is open, and since $k$ fixes $o$ it must be the case that $k$ acts trivially 
on all of $\widehat{\Omega}$ and thus be the identity, as desired.
\end{proof}

\begin{prop}\label{fix2}
 
The group $G$ acts freely on ${\rm id}_2(o)$, and thus on all of $J^{\nd}_k\widehat{\Omega}$ for $k\geq 2$.
 
\end{prop}
 
\begin{proof}
 
By Lemma \ref{fix1} we need only show that no element of $N^{-}$ fixes $\id_2(o)$. 
Let $N\subset N^-$ be the stabilizer of $\id_2(o)$ and note that $N$ is normal in $B$. 
Since $K_{\C}$ contains a maximal Cartan algebra,  the lie algebra  of $N$ must 
be a direct sum of root spaces. Assuming $N$ is not trivial, we let $\alpha$ be one 
of those roots, and $N_{\alpha}$ the corresponding one-dimensional subgroup. 
Now there is a map $\SL_2(\R)\ra G$ which sends ${\rm SO}(2)$ to $K$ and the roots of $K$ 
(which lie in $\SL_2(\C)$) to $N_{\alpha}$ and its conjugate $N_{-\alpha}$. 
This map induces a map of symmetric spaces $\Hh\ra \Omega$ and 
$\PP^1(\C)\ra \widehat{\Omega}$. 
Moreover, this map is evidently holomorphic. We have thus reduced the claim to the 
case of $\SL_2(\C)$ where it may be easily checked by hand, 
since $\frac{z}{1+az} = z-az^2+O(z^3).$
\end{proof}
 
Note that it follows from the proposition that $W_m$ is closed in $J^{\nd}_m$ for $m\geq2$.
 
\begin{cor}\label{3suffices}
 
If $F:\widehat{\Omega}\ra \widehat{\Omega}$ satisfies $J_3F\circ{\rm id}_3\subset W_3$ 
then there exists $g\in G(\C)$ such that $F(z)=gz$.

\end{cor}


\begin{proof}

By Proposition \ref{fix2} the variety $W_3$ projects bijectively onto $W_2$. Thus, in local coordinates, 
we may write equations for all the third order derivatives of $F$ in terms of its lower order
derivatives. Differentiating further gives us the full power series expansion of $F$ and thus characterizes 
it completely. This completes the proof.
\end{proof}

\subsection{Lowering the order of the differential equation}

In this section we improve over Corollary \ref{3suffices} in the following ways:

\begin{thm}\label{7.4}
 Let $n \ge 2$, $\Omega = \mathbb B^n$, $G(\mathbb C) = \mathbb P{\rm GL}(n+1;\mathbb C)$.  If $f: (\widehat \Omega;x_0) \to \ (\widehat \Omega;y_0)$ satisfies $J_2f\circ {\rm id}_2 \subset W_2$, then there exists $g \in G(\mathbb C)$ such that $f(z) = gz$ wherever $f$ is defined.
 \end{thm}

\begin{thm}\label{7.5}  Let $\Omega$ be an irreducible bounded symmetric domain of rank $\ge 2$, $G(\mathbb C) = {\rm Aut}(\widehat \Omega)$.  If $f: (\widehat \Omega;x) \to \ (\widehat \Omega;y)$ satisfies $J_1f\circ {\rm id}_1 \subset W_1$, then there exists $g \in G(\mathbb C)$ such that $f(z) = gz$ wherever $f$ is defined.
\end{thm}

Both Theorems \ref{7.4} and \ref{7.5} follow from known results, Theorem \ref{7.4} from a local version (for holomorphic maps) of the Fundamental Theorem of Projective Geometry in the case of the complex field, and 
Theorem \ref{7.5} from Ochiai's Theorem characterizing automorphisms of $S$-structures (Ochiai \cite{Oc70}) for $S$ being an irreducible Hermitian symmetric space of the compact type of rank $\ge 2$. They can be 
stated as 
follows.

\vskip 0.3cm
\noindent
{\bf Theorem A} \ {\it Let $n \ge 2$, $\Omega = \mathbb B^n$, $G(\mathbb C) = \mathbb P{\rm GL}(n+1;\mathbb C)$, $\mathbb B^n \Subset \mathbb C^n \subset \widehat {\mathbb B}^n = \mathbb P^n$ the standard embeddings. Suppose $U \subset \mathbb C^n$ is a convex open set and $f: U \overset\cong\longrightarrow V$ is a biholomorphism onto an open subset $V \subset \mathbb P^n$ such that, for any nonempty $($connected$)$ intersection $\ell \cap U$ of an affine line $\ell$ with $U$, we have $f(\ell \cap U) \subset \ell'$ for some affine line $\ell'$ on $\mathbb C^n$.  Then, $f$ extends to a biholomorphic automorphism $F: \mathbb P^n \to \mathbb P^n$, i.e., $F \in \mathbb P{\rm GL}(n+1;\mathbb C)$.}

\vskip 0.2cm
A proof of this is given in Mok [24
, Section (2.3)].  For the formulation of Ochai's Theorem recall that an irreducible Hermitian symmetric space of the compact type is given by $S = G/B$ where $G$ is a simple complex Lie group and $B \subset G$ is some maximal parabolic subgroup.  For any point $x \in S$ let $B_x \subset G$ be the isotropy (parabolic) subgroup at $x$.  Let $B_x = U\!\cdot\!L$ be the Levi decomposition of $B_x$ where $U \subset B_x$ is the unipotent radical and $L \subset B_x$ is a Levi factor.  As is well-known, for any $\upsilon \in U$, $d\upsilon(x) = {\rm id}_{T_x(X)}$, hence the map $\Phi(\gamma) = d\gamma(x)$ defines a representation of $L = B_x/U$ on $T_x(X)$ which is independent of the choice of the Levi decomposition. We denote by $\mathcal W_x \subset \mathbb PT_x(X)$ the highest weight orbit of the action of $L$ on $\mathbb PT_x(X)$.  Then, Ochiai's Theorem can be formulated as follows.

\vskip 0.3cm
\noindent
{\bf Theorem B (Ochiai \cite{Oc70})} \
{\it Let $S$ be an irreducible compact Hermitian symmetric manifold of the compact type
and of rank $\ge 2$; $U, V \subset S$ be connected open subsets, and 
$f: U \overset\cong\longrightarrow V$ be a biholomorphism.
Suppose for every $x \in U$ the differential $df(x): T_x(S) \to T_{f(x)}(S)$
satisfies $[df(x)](\mathcal W_x) = \mathcal W_{f(x)}$. 
Then, there exists $F \in \text{\rm Aut}(S)$ such that $F\big|_U \equiv f$.}

\vskip 0.2cm
Ochiai \cite{Oc70} made use of Lie algebra cohomology. Observing that $\mathcal W_x$ agrees with $\mathcal C_x(S)$, the variety of minimal rational tangents (VMRT) at $x \in S$ consisting of $[\alpha] \in \mathbb PT_x(X)$ tangent to minimal rational curves passing through $x \in X$, Hwang-Mok \cite{HM01} generalized Theorem B to Fano manifolds of Picard number 1, proving analogously the Cartan-Fubini extension theorem for VMRT-preserving germs of biholomorphisms between Fano manifolds of Picard number 1 under mild geometric conditions.  A differential-geometric proof of Theorem B was given in Mok [24, Sections (2.2)-(2.4)].

\vskip 0.3cm
{\bf Deduction of Theorem \ref{7.4} from Theorem A}

\noindent
{\it Proof of Theorem \ref{7.4}} \quad
Let $U \subset \mathbb C^n$ be a domain of definition of $f: (\widehat \Omega;x_0) \to \ (\widehat \Omega;y_0)$ which we may assume to be convex.  By assumption $J_2f\circ{\text id}_2 \subset W_2$.  In other words, for any $x \in U$, $J_2f(x) = J_2\gamma_x(x)$ for some $\gamma_x\in \mathbb P{\rm GL}(n+1;\mathbb C)$ such that $\gamma_x(x) = y := f(x)$.  Let $\ell \subset \mathbb P^n$ be a projective line passing through $x$, then $\gamma_x(\ell) =: \ell' \subset \mathbb P^n $ is a projective line passing through $y$, thus from the assumption $f(\ell)$ is tangent to $\ell'$ to the order $\ge 2$.

\vskip 0.15cm
On $\mathbb P^n$ there is a canonical projective structure, defined as follows. For each projective line $\ell$ on $\mathbb P^n$ denote by $\widehat \ell \subset \mathbb PT_{\mathbb P^n}$ the tautological lifting of $\ell$ obtained as the image of the map associating the pair $(\ell,t)$, $t \in \ell$, to $\mathbb PT_t(\ell) \in \mathbb PT_t(\mathbb P^n)$.  Then this gives a 1-dimensional (holomorphic) foliation $\mathcal F$ on the total space $\mathbb PT(\mathbb P^n)$ of the projective tangent bundle over $\mathbb P^n$. We call $\mathcal F$ the canonical projective structure on $\mathbb P^n$.  In the notation of the last paragraph, the germ $(f(\ell);y)$ is second order tangent to $(\gamma_x(\ell);y)$.  Lifting to $\mathbb PT_{\mathbb P^n}$, it means that the tautological lifting $\Lambda$ of the germ of holomorphic curve $(f(\ell);y)$ is tangent (to the order $\ge 1$) to the lifting $\widehat \ell'$ of the unique projective line $\ell' = \gamma_x(\ell)$ passing through $y$ and tangent to $f(\ell)$ at $y$.  Let $\ell_0 \subset \mathbb P^n$ be an arbitrary projective line passing through $x$.  We have proven that for some neighborhood $\mathcal V$ of $[df](\mathbb PT_{x_0}(\ell_0)) \in \mathbb PT_{y_0}(\mathbb P^n)$, the tautological liftings of images of projective lines under $f$ define a holomorphic 1-dimensional foliation $\E = f_*\mathcal F$ on $\mathcal V$ which is tangent to the canonical projective structure $\mathcal F$ at every point $[\alpha] \in \mathcal V$.  But $\mathcal F$ itself is a 1-dimensional foliation, hence  $\E$, where defined, agrees with $\mathcal F$.  This translates to the statement that $f(\ell\cap U) \subset \ell'$, and we may apply Theorem A to deduce that $f: U \overset\cong\longrightarrow V$ extends to a biholomorphic automorphism $F: \mathbb P^n \overset\cong\longrightarrow \mathbb P^n$, as desired. \quad \quad $\square$       

\vskip 0.3cm
\noindent
{\bf Deduction of Theorem \ref{7.5} from Theorem B} 

\noindent
{\it Proof of Theorem \ref{7.5}} \quad In analogy to the proof of Theorem \ref{7.4} we deduce Theorem 7.5 from Theorem B (Ochiai's Theorem) and the hypothesis $J_1f\circ{\rm id}_1 \subset W_1$ that given $x \in U$ and $\ell$ a projective line passing through $x$, $f(\ell \cap U)$ must be tangent to some projective line $\ell' = \gamma_x(\ell)$ at $y = f(x)$ for some $\gamma_x \in G$ such that $\gamma_x(x) = y$.  Thus $[df](\mathcal C(S)|_U) = \mathcal C(S)|_V$, and by Theorem B $f$ extends to a biholomorphic automorphism $F \in {\rm Aut}(S).$ $\square$

\vskip 0.3cm
\noindent
{\bf Theorem \ref{7.4} and Theorem \ref{7.5} when $\Omega$ is reducible}

\noindent
The analogue of Theorem \ref{7.4} holds for any bounded symmetric domain $\Omega$ provided that there are no irreducible factors biholomorphic to the disk $\mathbb B^1$.  Likewise, the analogue of Theorem \ref{7.5} holds for any bounded symmetric domain $\Omega$ provided that each irreducible factor of $\Omega$ is of rank $\ge 2$.  The proofs are small variations of the irreducible case. (For $\Omega = \Omega_1 \times \cdots \times \Omega_m$ one considers the $m$ moduli spaces of projective lines $\mathcal K_1, \cdots, \mathcal K_m$ of the compact duals $\widehat \Omega_1, \cdots, \widehat \Omega_m$.)

\section{Connection Formula and Automorphic Functions}

\subsection{Definition of Schwarzian variety}

Notice that since $N^+$ acts with an open orbit on $\widehat{\Omega}$, it follows that 

$$
G\backslash J_k\widehat{\Omega}\cong B\backslash J_k\widehat{\Omega}(o)
$$
birationally. We define this latter variety to 
be the  \emph{$k$'th Schwarzian variety} $S_k(\widehat{\Omega})$. Now for
any rational map $F: \widehat{\Omega}\ra \widehat{\Omega}$ 
we may define the {\it Schwarzian\/} of $F$ to be 
$$
S_k(F)(z) = G\cdot{J_kF(\id_k(z))}\in S_k(\widehat{\Omega}).
$$ 
at any point $z \in \widehat\Omega$ where $F$ is a morphism.  If we restrict to the upper half place 
and $k=3,$ then $S_3(\PP^1(\C))$ may 
be identified with $\PP^1(\C)$ and $S_k(F)$ becomes the usual Schwarzian.

We may define a rational map 
$$
\psi_k:S_k(\widehat{\Omega})\times (N^+\backslash J_k(\widehat{\Omega})) 
\ra S_k(\widehat{\Omega})
$$ 
as follows: an element of $J_k(\widehat{\Omega})$ corresponds naturally 
to a point $z\in \widehat{\Omega}$ and a germ
of a map $F: \widehat{\Omega}\ra \widehat{\Omega}$ sending $o$ to $z$. 
An element of $S_k(\widehat{\Omega})$ corresponds to the germ of a map 
$H:\widehat{\Omega}\ra \widehat{\Omega}$ sending $o$ to $o$ up to a left $B$-translation.
If $z \in N^+o$ we let $\nu_z\in N^+$ be the element sending $z$ to $o$, we may consider the 
composition $H\circ \nu_z\circ F$ as an element of  $S_k(\widehat{\Omega})$. 
It is clear that this is well defined.

For any positive integer $d$,we can  naturally extend this to a map
$$\psi_k:S_k(\widehat{\Omega})^d\times (N^+\backslash J_k(\widehat{\Omega}))^d
\ra S_k(\widehat{\Omega})^d$$

\subsection{Connection Formula}

For 2 maps $F_1,F_2:\widehat{\Omega}\ra \widehat{\Omega}$, it follows that 
$J_kF_1\circ J_kF_2=J_k(F_1\circ F_2)$, from which it follows that 
$$
\psi_k\Big(S_k(F_1)(F_2(z)) ,  J_kF_2(\id_k(z))\Big) = S_k(F_1\circ F_2)(z).
$$ 
We refer to this equation as  the {\it connection formula}.

\subsection{Schwarzians of Automorphic Functions}

Let $\Gamma\subset G$ be a discrete subgroup such that $X=\Gamma\backslash\Omega$ 
is a Shimura Variety.  Consider a rational embedding\footnote{Since the complex function field of the $n$-dimensional projective variety $X$ is always isomorphic to that of an irreducible hypersurface in $\mathbb P^{n+1}$ and since $\widehat{\Omega}$ is rational, such an embedding always exists, and we could deal with only a single function if $X$ were rational, as is the case of $X(1)$ and the $j$-function.} 
$\phi=(\phi_1,\phi_2):X\ra\widehat{\Omega}^2$, and let $p_i=\phi_i\circ q: \Omega\ra \widehat{\Omega}$. 
Set $p=(p_1,p_2)$.
Writing the connection formula for $p_i, p_i^{-1}$ gives:
$$
\psi_k\Big(S_k(p_i^{-1}(p_i(z)),J_kp_i(\id_k(z))\Big) = S_k(\id).
$$

Now, inverting the action yields the relation 
$$
S_k(p_i^{-1})(p_i(z)) = \psi_k\Big(S_k(\id),J_kp_i(\id_k(z))^{-1}\Big).\leqno{(*)}
$$

Since $(*)$ shows that $\left(S_k(p_1^{-1})(p_1z)),S_k(p_2^{-1})(p_2(z))\right)$ is definable on a fundamental domain, 
and it is clearly $\Gamma$ invariant, it must be a (single-valued) algebraic function of $q(z)$. 
Since $\phi=(\phi_1,\phi_2)$ is a rational embedding, we conclude the following:

\begin{prop}\label{rationals}

$\left(S_k(p_1^{-1})(\phi_1(x))),S_k(p_2^{-1})(\phi_2(x))\right)$ is an algebraic function $R_k(x)$ for $x\in X$.

\end{prop}


\begin{thm}\label{determinej}

Suppose that $U\subset\widehat{\Omega}$ be a connected open set, and  $F:U\ra \phi(X)\subset\widehat{\Omega}^2$ satisfies 
$$
\psi_3\Big(R_3(\phi^{-1}(F(z))), J_3F({\rm id}_3(z))\Big) = \left(S_3({\rm id}),S_3({\rm id})\right)
$$ 
Then there exists $g\in G(\C)$ such that $F(z)=p(gz)$. 

\end{thm}


\begin{proof}
Using the connection formula and Proposition \ref{rationals}, we see that
\begin{align*}
S(p^{-1}\circ F)&=\psi_3\Big(S_3(p^{-1})(F(z)), J_3F(\id_3(z))\Big) \\
&=\psi_3\Big(R_3(\phi^{-1}(F(z))),  J_3F(\id_3(z))\Big) \\
&=(S_3(\id),S_3(\id))
\end{align*}
But now Theorem \ref{3suffices} shows that $(p^{-1}\circ F)(z)=gz$ for some $g\in G(\C)$, 
which completes the proof.
\end{proof}

Note that Theorem \ref{7.4} shows that one may use $R_2(z)$ in 8.2
in the case where $\Omega = \mathbb B^n$, $n \ge 2$, and Theorem \ref{7.5} shows that one may use $R_1(z)$ in the case where $\Omega$ is irreducible and of rank $\ge 2$, with obvious generalizations to the reducible cases (cf.\,the last paragraph of \S7).

\section{Ax-Schanuel with Derivatives}

\begin{thm}\label{ASD}

Let $k\geq 2$ and $r\geq\dim X$ be positive integers. With notation as above, 
let $W\subset J^{\nd,r}_k\Omega\times J^{\nd,r}_kX$ be an algebraic subvariety. 
Let $U$ be an irreducible component of $W\cap J^{\nd,r}_kD$ whose 
dimension  is larger than expected, that is,
$$
\dim W < \dim U + \dim G. 
$$
Then the projection of $U$ to $X$ is contained in a proper weakly special subvariety of $X$.

\end{thm}

\bigbreak

We first need to state a lemma.

\begin{lemma}\label{graphclosure}

For $k\geq 2$, consider the graph 
$$
J^{\nd,r}_kD\subset J^{\nd,r}_k\Omega\times J^{\nd,r}_kX
$$
of the projection map $J^r_kq:J^{\nd,r}_k\Omega\ra J^{\nd,r}_kX$. 
Then $G(\C)\cdot J^{\nd,r}_kD$ is a closed, algebraic subvariety, 
and is the Zariski closure of $J^{\nd,r}_kD$.

\end{lemma}

\begin{proof}

First, note that $J^{\nd,r}_kD$ is invariant under $\Gamma$, and therefore its Zariski closure 
is invariant under $G(\C)$. Thus, it is sufficient to show that $Y := G(\C)\cdot J^{\nd,r}_kD \subset J^{\nd,r}_k\widehat{\Omega}\times J^{\nd,r}_kX$ is a (closed) algebraic subvariety. 	


We claim that $Y \subset J^{\nd,r}_k\widehat{\Omega}\times J^{\nd,r}_kX$ is constructible complex analytic.  Consider the holomorphic map $\Psi_0: G(\C) \times J^{\nd,r}_k\widehat{\Omega} \to J^{\nd,r}_k\widehat{\Omega}$ defined by $\Psi_0(\gamma,\omega) = \gamma\cdot\omega$.  Since $G(\C)$ acts algebraically on $X$, there exist projective compactifications\footnote{We postpone to 9.1 the construction of a projective compactification of $J^{r}_k\widehat{\Omega}$ as a bundle of weighted projective spaces.} $G(\C)'$ of $G(\C)$ and $(J^{\nd,r}_k\widehat{\Omega})'$ of $J^{\nd,r}_k\widehat{\Omega}$ such that $\Psi_0$ extends to a rational map $\Psi_0' : G(\C)' \times (J^{\nd,r}_k\widehat{\Omega})' \to (J^{\nd,r}_k\widehat{\Omega})'$.
Write $\Psi = \Psi_0\times {\rm id}$, i.e., $\Psi:  (G(\C) \times (J^{\nd,r}_k\widehat{\Omega}))\times J^{\nd,r}_kX \to J^{\nd,r}_k\widehat{\Omega}\times J^{\nd,r}_kX$ is given by $\Psi(\gamma,\omega,\nu) = (\gamma\cdot\omega,\nu)$. Define $\Psi' := \Psi_0' \times {\rm id}: (G(\C)' \times (J^{\nd,r}_k\widehat{\Omega})')\times J^{\nd,r}_kX \to (J^{\nd,r}_k\widehat{\Omega})'\times J^{\nd,r}_kX$.  Applying the proper mapping theorem to the graph of the restriction of $\Psi'$ to $G(\C)' \times J^{\nd,r}_kD \subset G(\C)' \times ((J^{\nd,r}_k\widehat{\Omega})'\times J^{\nd,r}_kX)$ and noting that $Y$ is $\Gamma$-invariant, we deduce that $\overline Y \subset J^{\nd,r}_k\widehat{\Omega}'\times J^{\nd,r}_kX$ is a complex analytic subvariety, hence
$Y \subset \overline{Y}$ is constructible complex analytic.  

Next, we argue that $Y \subset J^{\nd,r}_k\widehat{\Omega}\times J^{\nd,r}_kX$ is a closed subset.  Observe by Proposition \ref{fix2} that $G(\C)$ acts freely on $J^{\nd,r}_k\widehat{\Omega}\times J^{\nd,r}_kX$ so that any $G$-invariant constructible complex analytic subset of $J^{\nd,r}_k\widehat{\Omega}\times J^{\nd,r}_kX$ is of dimension $\ge \dim G$.
Now, note that $Y$ is invariant under the automorphisms $A$ of the disc $\D^{r}_{k}$ 
and under the action of $G(\C)$. Moreover, these actions commute 
and $A\times G(\C)$ acts transitively on $J_k^{\nd,r}\widehat{\Omega}$. 
It follows that if we let $Z=\overline{Y}-Y$ be the 
boundary of $Y$, then $Z$ is equidimensional over the first factor  $J^{\nd,r}_k\widehat{\Omega}$. 
However, the pullback of $Y$ to $J_k^{\nd,r}\widehat{\Omega}\times J_k^{\nd,r}\widehat{\Omega}$ is symmetric,
and so $Z$ is also equidimensional over the second factor $J_k^{\nd,r}X$. 
Now since $\dim Z<\dim Y$ it follows that the dimension of the fibers of 
$Z$ over the second factor $J^{\nd,r}_k X$ are less 
than $\dim G$, which means they are empty since $Z$ is closed under the action of $G(\C)$,
proving the claim that $Y \subset J^{\nd,r}_k\widehat{\Omega}\times J^{\nd,r}_kX$ is closed.

Since $Y \subset J^{\nd,r}_k\widehat{\Omega}\times J^{\nd,r}_kX$ is closed and constructible complex analytic, it must be a complex analytic subvariety.
On the other hand, since $Y$ is given by $G(\C)$ acting on the restriction of $J^{\nd, r}_kD$ to $J^{\nd, r}_k\mathcal{F}$ it follows that it is definable, and by Theorem \ref{defchow} $Y$ is an algebraic subvariety, as desired.
\end{proof}

\begin{cor}

Let $z_1,\dots,z_n$ be an $N^+(\C)$-invariant
algebraic coordinate system on $\Omega$. 
Let $\{\phi_1,\dots,\phi_N\}$ be a $\C$-basis of modular functions.
Then the field generated by $\{\phi_i\}$ and their partial derivatives
with respect to the $z_j$ up to order $k\geq 2$ has transcendence 
degree over $\C$ equal to $\dim G$. Further, the transcendence degree is the same 
over $\C(z_1,\dots,z_n)$.

\end{cor}

\begin{proof}

From Lemma \ref{graphclosure} applied to $r=\dim X = n$, it follows that 
$${\rm tr.deg\/}_{\C}\C(\{z_i\}, \{\phi_j^{(\nu)}\}) = \dim G+n.$$ 
The algebraic independence will therefore follow as soon as we show  
that the transcendence degree of $\C(\phi_1,\dots,\phi_N)$ is equal to $\dim G$. 

To see this, consider 
$$
V=({\rm id}\times q)(\id_k(\Omega))\subset J^{\nd,n}_k D.
$$
We have to show that the Zariski-closure $W$ of the projection of $V$ to
$J^{\nd,n}_kX$ has dimension $\dim G$. 
Let $p\in X$ be a point, and, without loss of generality, let $o\in \Omega$ be a pre-image. 

Identifying 
$N^+(\C)\backslash J^{\nd, \n}_k(N^+o)$  with $J^{\nd,n}_k(o)$ by quotienting out by the action of 
$N^+(\C)$, we get a rational map $\psi:G(\C)\ra J^{\nd,n}_k(o)$ by $\psi(g) = g\cdot \id_k(o)$. 
Now, the pre-image of $p$ in $W$ contains $\psi(\Gamma)$ so it must contain $\psi(G(\C))$.

It follows that $W$ contains $q\big(G(\C)\cdot\id_k(o)\big)$, 
hence by definability must be equal to its 
closure (definable since we can restrict to a fundamental domain, and its closure is complex analytic
since it is algebraic in $\widehat{\Omega}$).  Since $G(\C)$ acts freely 
on $J^{\nd,n}_k\widehat{\Omega}$ 
it follows that $\dim G(\C)\cdot\id_k(o) = \dim G$, as desired.
\end{proof}

\begin{remark}

The same argument shows that the projection
$$J_{k+1}q\left(G\cdot J^{\nd,n}_{k+1}({\rm id}_{k+1}(o))\right)\rightarrow
J_kq\left(G\cdot J^{\nd,n}_{k}({\rm id}_k(o))\right)
$$
is bijective for $k\geq 2$, and that
$J_kq\left(G\cdot J^{\nd,n}_{k}({\rm id}_k(o))\right)$ 
is the Zariski closure of the graph of the $\nu$'th partial derivatives of $q$ for $|\nu|\leq k$. 
It follows that, for $k\geq 2$, the $k$-th partial derivatives of $q$
are rational in the $\nu$'th partial derivatives of $q$ for $|\nu|\leq 2$. 
In other words, the field generated by all the partial derivatives of $q$ 
is generated by the partial derivatives of order $\leq 2$.

\end{remark}

\subsection{Compactifying Jet spaces}

We shall require a compactification of Jet spaces to discuss Hilbert schemes. 
Thus, we define 
$$BJ^r_kY:=\Hom(\D^r_k,Y)\times \mathbb{A}^1_{\C} / \mathbb{G}_{\rm m,\C}$$ 
where the action is defined via
$r\cdot(t\mapsto f(t),s):=(t\mapsto f(rt),rs)$. It is easy to see, 
by expanding into  local coordinates given by 
Taylor series 
coefficients, that $BJ^r_kY$ is a weighted projective space over $Y$, and is thus 
a projective variety if $Y$ is projective. This then gives a functorial compactification of $J^r_kY$.

\subsection{Descending Hilbert scheme loci}\label{hilbert}

Now we fix some algebraic subvariety $W\subset J^{\nd,r}_k\Omega\times J^{\nd,r}_kX$, 
with  $\widehat{W}\subset BJ^r_k\widehat{\Omega}\times BJ^r_k\widehat{X}$
its Zariski closure, and $U$ an irreducible component of $W\cap J^{\nd,r}_kD$.
We make no assumptions here on the dimension of $U$.

Let $M$ be the Hilbert 
scheme of all subvarieties of $BJ^r_k\widehat{\Omega}\times BJ^r_k\widehat{X}$ with  Hilbert polynomial $P$.
Then $M$ also has the structure of an algebraic variety. 
Corresponding to $y\in M$ we have the subvariety $W_y\subset J^{\nd,r}_k\Omega\times J^{\nd,r}_kX$,
and we have the incidence variety (universal family)

\[
B=\{(z,x,y)\in J^{\nd,r}_k\Omega\times J^{\nd,r}_kX\times M: (z,x)\in W_y\},
\]
and the family of the intersections of its fibres over $M$ with $J^{\nd,r}_kD$, namely
\[
A=\{(z,x,y)\in J^{\nd,r}_k\Omega\times J^{\nd,r}_kX\times M: (z,x)\in W_y\cap J^{\nd,r}_kD\}.
\]
Then $A$ is a closed complex analytic subset of $J^{\nd,r}_k\Omega\times J^{\nd,r}_kX\times M$. 
It has natural projection $\theta: A\rightarrow M$, with $(z,x,y)\mapsto y$. 
Then, for each natural number $k$, the set
\[
A(k)=\{(z,x,y)\in J^{\nd,r}_k\Omega\times J^{\nd,r}_kX\times M: \dim_{(z,x)}\theta^{-1}\theta(z,x,y)\ge k\},
\]
the dimension being the dimension at $(z,x)$ of the fibre of the projection in $A$,
is closed and complex analytic see e.g. the proof of \cite{PSNEWTON}, Lemma 8.2,
and references there.

Now we have the projection $\psi: J^{\nd,r}_k\Omega\times J^{\nd,r}_kX\times M\rightarrow J^{\nd,r}_k\Omega\times J^{\nd,r}_kX$,
and consider$$Z=Z(k)=\psi(A(k)).$$
Then as $M$ is compact, $\psi$ is proper and so $Z$ is closed in $J^{\nd,r}_k\Omega\times J^{\nd,r}_k X$. 
Note that $Z$ is $\Gamma$-invariant
and $Z\cap J^{\nd,r}_k\FF\times J^{\nd,r}_kX$ is definable.

\begin{lemma}\label{algTD}
Let $T=(q\times{\rm id})(Z)$. Then $T\subset J^{\nd,r}_kX\times J^{\nd,r}_kX$ is closed 
and algebraic.
\end{lemma}

\begin{proof}
The same as Lemma \ref{algT}.
\end{proof}

\medskip

%
%

\bigbreak

\section{Proof of Theorem \ref{ASD}}

\bigskip

\begin{proof}
We argue by induction, in the first instance (upward) on $\dim\Omega$.
For a given $\dim\Omega$, we argue (upward) on $\dim W-\dim U$.
We then argue by induction (downward) on $\dim U$, and finally upward on $r$.

We  carry out the constructions of \S\ref{hilbert} with $k=\dim U$ and keep the notation there. 
We let $A(k)'\subset A(k)$  be the irreducible component which contains $U\times [W]$, 
and $Z'=\psi(A'(k))\subset Z$ be the
corresponding irreducible component of $Z$, and $V=(q\times{\rm id})(Z')$ the irreducible 
component of $T$, which is therefore algebraic by Lemma \ref{algT}.
Now, by assumption $V$ contains $q(U)$, and so it is not contained in any proper weakly special of the diagonal $J^{\nd,r}_k\Delta_X$, and thus has Zariski-dense monodromy.

\medskip

Consider the family  $F_0$ of algebraic varieties corresponding to $A(k)'$. 
Let $\Gamma_0\subset\Gamma$ 
be the subgroup of elements $\gamma$ such that a very general member of $F_0$ 
is invariant by $\gamma$. 
Note that a very general element $W'$ of $F_0$ is  invariant by exactly the subset $\Gamma_0$ of 
$\Gamma$. Let $\Theta$ be the connected component of the 
Zariski closure of $\Gamma_0$ in $G(\R)$. 

\medskip

\begin{lemma}

$\Theta$ is a normal subgroup of $G(\R)$. 

\end{lemma}

\begin{proof}

This is proven exactly as in Lemma \ref{normalsub}.
\end{proof}
%
%

\begin{lemma} $\Theta$ is the identity subgroup.

\end{lemma}

\begin{proof}

We argue by contradiction.  Without loss of generality we may assume 
that $W$ is a very general member of $F_0$, and hence is invariant by exactly $\Theta$.
Since $\Theta$ is a $\Q$-group by construction, 
it follows that $G$ is isogenous to $\Theta\times\Theta'$ and we have a 
splitting of Hermitian symmetric domains $\Omega=\Omega_{\Theta}\times\Omega_{\Theta'}$. 
Replacing $\Gamma$ by a finite index subgroup we 
also have a splitting $X\cong X_{\Theta}\times X_{\Theta'}$. Moreover, $D$
splits as $D_{\Theta}\times D_{\Theta'}$.

By our induction on $\dim W$, it follows from Lemma \ref{graphclosure} that
$W\subset G(\C)\cdot J^{\nd,r}_kD$.  Now, we let  
$$
W_1\subset J^{\nd,r}_k\Omega_{\Theta'}\times J^{\nd,r}_kX_{\Theta'}\times J^{\nd,r}_kX_{\Theta}
$$
be the projection of $W$. Let $U_1$ be the projection of $U$ to $W_1$. 
Since the map from $D_{\Theta}$ to $X_{\Theta}$ has discrete pre-images, 
it follows  that $\dim U = \dim U_1$. 

Now let $W'$ be the projection of $W_1$ to $\Omega_{\Theta'}\times X_{\Theta'}$ and  
$U'$ be a component of $W'\cap D_{\Theta'}$. Since $W\subset G(\C)\cdot J^{\nd,r}_kD$ 
and $W$ is closed under $\Theta(\C)$ we  see that $U'$ is the projection of $U$ to 
$\Omega_{\Theta'}\times X_{\Theta'}$. Now let $W''$ be the Zariski closure of $U'$. 
It follows by induction that $$\dim U'+\dim\Theta'\leq \dim W''.$$

Now for the projection map $\psi: W_1\ra W'$, the generic fiber dimension of $W''$ 
is the same as the generic fiber dimension over $U'$, and thus 
$$\dim U_1 + \dim\Theta'\leq \dim \psi^{-1}(W'')\leq \dim W_1.$$ 
Since $W$ is invariant under $\Theta(\C)$ it follows that  
$$\dim U +\dim G \leq \dim W$$ as desired.
\end{proof}

It follows that $W$ is not invariant by any infinite subgroup of $\Gamma$.  
The following lemma thus reaches a contradiction, and completes the proof. 

\begin{lemma}

$W$ is invariant by an infinite subgroup of $\Gamma$.

\end{lemma}

\begin{proof}

This is proven exactly as Lemma \ref{infinitesub}.
 \end{proof}
 
 \noindent
This contradiction completes the proof of Theorem \ref{ASD}.
\end{proof}


As a corollary,  we have the following concrete statement about transcendence 
degrees of modular functions and their derivatives on analytic subvarieties.

\begin{cor}\label{trdeg}

Let $V\subset \Omega$ be an irreducible complex analytic variety, not contained in a proper 
weakly special subvariety. Let  $\{z_i, i=1,\ldots, n\}$ be an algebraic 
coordinate system on $\Omega$. Let $\{\phi_j^{(\nu)}\}$ consist of a basis $\phi_1,\ldots, \phi_N$ of 
modular functions, all defined at at least one point of $V$,
together with their partial derivatives with respect to the $z_j$ 
up to order $k\ge 2$. Then
$$
{\rm tr.deg.}_{\C}\C\big(\{z_i\}, \{\phi_j^{(\nu)}\}\big) \geq \dim G + \dim V
$$	
where all functions are considered restricted to $V$.

\end{cor}

\begin{proof}

Let 
$$
U = (\id_2\times q\circ \id_2)(V)\subset J_2\Omega \times J_2 X,$$
so that $U$ is an analytic 
subvariety of the diagonal $J_2D$.  Note that $U$ does not record any of the differential 
information concerning $V$, but instead a coordinate system for $U$ is given by 
$z_1,\dots,z_n, \phi_1,\dots,\phi_N$. The result now follows immediately 
from Theorem \ref{ASD} applied to the Zariski closure of $U$.
\end{proof}

\bigbreak

\centerline{III. Ax-Schanuel in a differential field\/}

\bigskip

In this part we formulate a version of Ax-Schanuel in the setting of a differential field.
We further show that the jet version (\ref{ASD}) 
may be deduced directly from the differential version (12.3).

\section{Characterizing the uniformization map}

We would like to have a criterion in a differential field to determine when a pair of maps 
$$
w:\Delta\ra \widehat{\Omega},\quad u: \Delta\ra X=\Gamma\backslash \Omega,
$$
where $\Delta$ is a disk of given dimension, satisfies $u=q(gv)$ for some $g\in G(\C)$.
To this end, we use Schwarzian varieties.

\begin{lemma}\label{generaldimdiffeq0}

Let $k$ be a positive integer.
There exists a positive integer $m=m(G, q, k)$ with the following property.
Let $\Delta^{k}$ be a $k$-dimensional disk, and consider a pair of maps 
$(w,u): \Delta^k\ra \widehat{\Omega} \times \widehat{\Omega}$. 
If at every $t\in \Delta^k$ there exists 
$g_t\in G(\C)$ such that $u(t)=g_tw(t)$ to order $m$,  
then there exists a global $g\in G(\C)$ with $u(t)=gw(t)$.

\end{lemma}

\begin{proof}

Let $d=\dim G$. For positive integers $a\geq b$ we define 
$$
X_{a,b}
= \{(\phi,\psi)\mid \exists\gamma\in G(\C), \gamma\circ\phi=\psi\textrm{ to order } b \}
\subset J^k_a\widehat{\Omega}\times J^k_b\widehat{\Omega}.
$$

We partition $X_{d, d}$ into subsets $X^i_{d, d}$ consisting of all pairs $(\phi,\psi)$ 
where $\stab\phi_i=\stab\phi_{i+1}$, where $\phi_i$ denotes
$\phi$ to the $i$'th order, and $i$ is the smallest such integer.
Note that this is indeed a partition since each time $\stab\phi_i\neq \stab\phi_{i+1}$ 
the dimension of $\stab\phi_{i+1}$ is at least 1 smaller than that of $\stab\phi_i$. 
Without loss of generality, assume that $(w, u)$
has image generically in $X^i_{d,d}$. 

Now, for all elements in $X^i_{d, d}$ it follows that there exists a universal equation for all 
partial $(i+1)$'st partial derivatives of $\psi$ in terms of the degree $\leq d$ derivatives of $\phi$ 
and the degree $\leq i$ partial derivatives of $\psi$ at o.
Thus one can solve for $u$ given $w$  and all the 
partial derivatives of degree $\leq i$ of $u$ at $0\in \Delta^k$. Clearly  $u(t)=g_o^{-1}w(t)$ 
is one such solution, and so that must be the only solution. This completes the proof.
\end{proof}

With $k, m$ as above let $J^k_mD\subset J^k_m\widehat{\Omega}\times J^k_m X$ 
be the graph of the projection morphism. 
Now let $V^k_m = G(\C)\cdot J^k_mD$, 
where the group $G$ acts only on the factor $J^k_m\widehat{\Omega}$.
Note that $V^k_m$ is definable, since 
$$
V^k_m=G(\C)\cdot (J^k_mD\vert_{ J^k_m\FF\times J^k_m X}).
$$
Also, $V^k_m$ is the image of $X_{m,m}$ under the projection map on the second factor, 
so its closure is analytic and 
of the  same dimension. It follows from Definable Chow that $V^k_m$ is constructible algebraic.

If now $w: \Delta^k\rightarrow \widehat{\Omega}$ is non-degenerate and $g\in G(\C)$
then the image of $(w, q(gw))$ in $J^k_m\widehat{\Omega}\times J^k_m X$ is contained in $V^k_m$.
We show the converse.

\begin{thm}\label{generaldimdiffeq}

Let $k$ be a positive integer and let 
$m=m(G, q, k)$ and $V^k_m$ be as above.
Let $(w,u):\Delta^k\ra \widehat{\Omega}\times X$. 
If the image of $J^k(w,u)$ is contained in 
$V^k_m$ then then there exists a global $g\in G(\C)$ with $u(t)=q(gw(t))$.

\end{thm}

\begin{proof}
Suppose $(w,u)$ lands in $V^k_m$. Let $\tilde{u}(t)$ be a local lift of $u(t)$ 
such that $u(t)=q(\tilde{u}(t))$. 
By Lemma \ref{generaldimdiffeq0} it follows that
there exists some $g_0\in G(\C)$ such that $w(t)=g_0\tilde{u}(t)$. 
Thus, $u(t)=q(g_0^{-1}w(t))$ as desired.
\end{proof}

\subsection{Uniformized Loci}

Given a 
map $w:\Delta^k\ra \widehat{\Omega}$, an element $g\in G(\C)$,  
and an integer $r\ge 2$,
we get in a natural way a map 
$$
L(w,g,r):\Delta^k\rightarrow \widehat{\Omega} \times X \times J^n_{r} X,
$$
where the second map is $q(gw)$ and the
third map is $J^n_r\big(q\circ g\big)\circ \id_r\circ w$,
which records the partial derivatives of $q_g$,
where $q_g(z)=q(gz)$, to order $|\nu|\le r$, restricted
to the image of $w$.
Such a map we will call a {\it uniformized locus.\/}
Note that the second map is repeated in the zero-order terms of the third map, and so is superfluous in a way, but we find it convenient to keep track of.

From such a map we obtain an image in the jet spaces (to some order $m$)
$$
J^*_mL(w,g,r)=\widetilde{v^g_k}:\Delta^k\rightarrow
J^k_{m}\widehat{\Omega} \times J^k_{m} X \times J^n_{r} X,
$$
where the first map is $J^k_m(w)\circ \id_m$,
the second map is $J^k_m\big(q\circ g\circ w \big)\circ\id_m$,
and the third map is (again) $J^n_r\big(q\circ g\big)\circ \id_r\circ w$.

We want differential equations which characterize when a trio of maps
$$
(w,v,u): \Delta^k\rightarrow \widehat{\Omega} \times X \times J^n_{r} X
$$
arising in this way. For $(w,u)$ this is dealt with by 11.2.

Now, we cannot directly talk about the map $q$, though we do have 
access to the map to $J^n_{m}X$. There is a complication when $w$ has a stabilizer, 
in that we could  replace $w$ by $g\circ w$ for any $g$ in the stabilizer, 
and in fact by a different, holomorphically 
varying, $g$ at every point. This would not affect the map or any of its derivatives, but it would affect 
the restrictions of the derivatives of $q\circ g$. Thus, we will equip the third coordinate with an 
extra differential equation to insist that the $g$ `stays constant'. We do this as follows.

We have an algebraic map  
$G(\C)\times J^n_{\ell} \widehat{\Omega}\ra J^n_{\ell} \widehat{\Omega}$ 
given by the natural action. 
Let $V'_{n, {\ell} }$ be the image of $G(\C)\times \id_{\ell} (\widehat{\Omega})$
restricted to $J^n_{\ell}\Omega$, 
and $V_{n, {\ell} }$ its constructible
algebraic (by Corollary 3.2) image in $J^n_{\ell} X$. 
Note that for ${\ell} \geq 3$ the variety $V'_{n,{\ell} }$ is fibered by varieties 
$\Omega$ over $G(\C)/N(\C)$.
If we descend to $X$, we lose the fibration but we retain an algebraic foliation. 

To make this precise, 
let $W'_{n,{\ell} }$ be the restriction to $TV'_{n,{\ell}}$ of the image of the natural holomorphic map 
$G(\C)\times T\id_{\ell} (\widehat{\Omega}) \ra TJ^n_{\ell}\widehat{\Omega}.$ 
Then the image $W_{n,{\ell} }$  of $W'_{n,\ell}$ in $TJ^n_{\ell}X$ determines an integrable
algebraic (by Corollary 3.2) foliation of $V_{n,{\ell} }$.
%

%

A map $L(w,g,r)$ has its image in $V_{k,m}\times V_{n,r}$ and the tangent
lands in $W_{n,r}$. We show that with these properties
characterize such maps.

\begin{thm}
Let $k$ be a positive integer, $r\ge 2$, and $m$ as above.
Consider a map 
$$
(w, u, v):\Delta^k\ra \widehat{\Omega} \times X \times J^n_{r} X.
$$
If the image of $(w,u,v)$ lands in $V_{k,m}\times V_{n,r}$ and the image of $T v$ 
lands in $W_{n,r}$, and $v$ restricts to $u$ then there exists $g\in G(\C)$
such that $(w,u,v)=L(w,g,r)$.
\end{thm}

\begin{proof}
Suppose that hypotheses are satisfied. By Theorem \ref{generaldimdiffeq} it follows that there 
exists $g\in G(\C)$ with $u=q\circ g\circ w$. So it remains to address the third 
coordinate. Note that if $w$ has no stabilizer, the claim follows immediately.
As it stands, the third map must be of the form $t\ra q\circ g(t)\circ\id_k\circ w$ 
for some function $g:\Delta^k\ra G(\C)$. However, by assumption $Tv$ lands in $W_{n, r}$ 
which is an integrable foliation whose leaves are precisely the set where $g(t)\in gN$. Thus we may write $g(t)=gn(t)$ for
a function $n(t)\in N$. However, since $v$ restricts to $u$ it follows that $n(t)$ is the identity function.
Thus the function $g(t)=g$ must be constant and the claim is proved.
\end{proof}


\section{Ax-Schanuel in a differential field}

\subsection{\bf The setting\/}

We fix $G$ and $q: \Omega\rightarrow X$. Let $n=\dim X$.
We take a field of definition $L_0\subset \C$, of finite type, for $X$ and for
the system of differential equations satisfied by $q$.

The weakly special subvarieties of $X$ come in countably many families,
and so correspond to points in suitable (possibly constructible)
subvarieties of countably many Hilbert schemes. These families are defined over 
$\overline{\Q}$ and the collection of families is stable under Galois automorphisms. 
So we may take them to be (not necessarily irreducible but)
defined over $L_0$. 

We consider a differential field $(K, \mathcal{D}, C)$. 
Here $\mathcal{D}=\{D_1,\ldots, D_d\}$ is a finite set of commuting derivations  
and $C$ is the constant field. 

It is a deep theorem of Kazhdan \cite{KAZHDANA, KAZHDANB, KAZHDANC}; 
see also Milne \cite{MILNEKAZHDANA, MILNEKAZHDANB}
that a conjugate of an arithmetic variety is again arithmetic.
It will therefore be important to ensure that our differential fields are identified
correctly with the complex object.
We therefore assume (initially) that $C$ contains a subfield $\Lambda_0$ isomorphic to $L_0$
under $\iota_0:\Lambda_0\rightarrow L_0$.

We take further a field of finite type $L\subset\C$ which is a field of definition for 
the constructible algebraic varieties in \S11 and assume that $C$ contains a field
$\Lambda$ isomorphic (by $\iota$) to $L$. Then we can identify the various varieties
in \S11 with their corresponding varieties in $K$.

\subsection{Rank and transcendence degree}

Let $V$ be an algebraic variety \footnote{or a scheme of finite type} over the constant field $C$, and let $p\in V(K)$ be a $K$-point. Let $U\subset V$ be an open
affine set defined over $C$ which contains $p$. Let $R=\mathcal{O}(U)$ be the ring of functions on $U$, and let $S$ be the image of $R$ in $K$ induced by evaluation on $p$. 
We define the \emph{transcendence degree} of $p$ to
be the transcendence degree of the fraction field of $S$ over $C$.

We define the \emph{rank} of $p$ to be the rank of the matrix $\rank(p)=(D_is)_{\substack{1\leq i\leq n \\ s\in S}}$.

Finally, given two varieties $V_1,V_2$ and points $p_1\in V_1(K),p_2\in V_2(K)$ we say that $p_2$ {\it is a function of \/} $p_1$ if $\rank(p_1,p_2)=\rank(p_2)$.

\subsection{\bf Statement and proof of Differential Ax-Schanuel\/}

{We consider $K$ points $(z,x,y)$ of $\widehat{\Omega}\times X\times J^n_r X$
where $r\ge 2$.

We assume that $x$ and $z$ have the same {\it rank\/},  and that $x$ is a function of $z$.
Set $k={\rm rank\/}(x)={\rm rank}(z)$, and assume $k\ge 1$
(for $k=0$ our theorems are true but trivial).

If $(z,x,y)$ satisfies the differential conditions corresponding to the hypotheses of  11.3 
then, under any Seidenberg embedding
{\it over $\iota$}, meaning extending 
$\iota$, as may always be assumed (see the version given in Scanlon \cite{SCANLON}),
we get tuples of regular functions
$$
(\overline{z}, \overline{x}, \overline{y}): \Delta^k\rightarrow \widehat{\Omega}\times X\times J_rX
$$
which give a uniformized locus.
It is thus a natural abuse of notation to refer to a tuple $(z,x,y)$ satisfying these conditions
as a {\it uniformized locus\/} in $K$.
We will say similarly that $x$ is {\it contained in a weakly special subvariety\/} if
it gives a $K$ point of one of the varieties defining the weakly special families.\/

We can now state a differential version of Ax-Schanuel. 

\medskip
\noindent
{\bf 12.3. Theorem.\/} (Differential Ax-Schanuel) 
{\it Fix $G, q, L$ as above. Let $(K, \mathcal{D}, C)$ be a differential field
with $\Lambda\subset K, \iota$ as above.
Let $(z, x, y)$ be a uniformized locus. Then
$$
{\rm tr. deg.}_CC(z, x, y)\ge {\rm rank}(z) +\dim G
$$
unless $x$ is contained in a proper weakly special subvariety.}

\begin{proof}
Suppose the transcendence degree ${\rm trd}_CC(z, x, y)$ is less, so that
there is a variety $W$ defined over $C$ containing these quantities
with $\dim W < {\rm rank}(z)+\dim G$.
Take a suitable finitely  generated differential $K'\subset K$
field containing $\Lambda$ and all constants appearing in the algebraic dependencies, the
$z$ and the associated $q$-quantities, and a field of definition of $W$.
Take a Seidenberg embedding of $K'$ over $\iota$ into a field of meromorphic functions
of $t\in \Delta^k$ where $k={\rm rank}(z)$. 

Then $\overline{x}=q_g(\overline{z})=Q(\ol{z})$ for some $g\in G(\C)$ for which
the partial derivatives of $Q$,
restricted to the image of $\overline{z}$ as a function of $t\in \Delta^k$, agree with the Seidenberg 
embeddings $\overline{y}$ of $y$ (including $Q=\overline{x}$).
Let $\Omega_g=g^{-1}\Omega$ be the domain of $Q$.

Let $U'$ be the locus in $J^n_{2}\Omega_g\times J^n_{2}X$ 
which is the graph under $J^nQ$ of the locus
$$
(\overline{z}, 1_{n\times n}, 0,\ldots, 0).
$$
Then $U'$ is the locus (i.e. over $t\in\Delta^k$) described by
$$
(\overline{z}, 1_{n\times n}, 0,\ldots, 0; \overline{x}, \overline{y})
$$
and is clearly in $J^{{\rm nd}, n}_{2}\Omega_g\times J^{{\rm nd}, n}_{2} X$,
(recall the assumption that $q$ is unramified).

We let $W'$ be the Zariski closure (over $\C$) of $U'$ and $U\subset W\cap J^n q_g$
the component containing $U'$. We have $\dim U'=k={\rm rank}(z)$,
while $W'$ is a subvariety of the image of $W$, so
$\dim_{\C}W'\le \dim_C W$. Hence we have
$$
\dim U\ge \dim U' = {\rm rank}(k) > \dim_{\C}W - \dim G \ge \dim_C W'-\dim G.
$$

Hence by Jet Ax-Schanuel for the map $Q$, which is just the same statement 
as for $q$,
we conclude that $\overline{x}$ is contained in a proper weakly special subvariety of $X$.
But then $x$ also has this property. 
\end{proof}

We next show that the jet version of Ax-Schanuel (Theorem 9.1) may in fact be deduced directly
from Theorem 12.3.

\subsection
{\bf Direct proof of 9.1 from 12.3\/} As in \cite{PTAS}.
We assume 12.3 holds.

If $A=\{f_1,\ldots, f_\ell\}$ is a set of regular functions of  
$t\in\Delta^k$, we set
$$
\dim A=\dim (f_1,\ldots, f_M)=
\dim \{\big(f_1(t),\ldots, f_M(t)\big): t\in\Delta^k\},
$$
where $\{\big(f_1(t),\ldots, f_M(t)\big): t\in\Delta^k\}$ is the locus parameterized by $A$.
The transcendence degree 
${\rm tr.deg.\/}_{\C}\C(f_1,\ldots, f_M)$ is the dimension of the Zariski closure
of this locus, which we denote $\dim{\rm Zcl}(A)=\dim{\rm Zcl}(f_1,\ldots, f_M)$.

We consider a locus $U\subset J^{{\rm nd}, n}_{\ell}D$, of dimension $k$ say,
where $\ell\ge 2$, meaning the graph of $q$ on some locus of non-degenerate jets.

We take $z=(z_1,\ldots, z_n), x=(x_1,\ldots, x_N)$ as affine coordinates on $\Omega$ 
and an open affine subset of $X$ containing an open subset of $U$.
We assume that $x_1,\ldots,x_n$ are algebraically independent on $X$ 
and the further variables dependent upon them.
We take coordinates
$$
(z, r, s, \ldots; x, R, S, \ldots)
$$
in $J^n_\ell \Omega\times J^n_\ell X$, where
$r=(r^i_j, i,j=1,\ldots, n)$ with $r^i$ representing the coordinates of the first
derivatives of $z_i$, likewise $R=(R^i_j)$ for $x_i, i=1,\ldots, N$, $s=(s^i_{jk})$ and 
$S=(S^i_{jk})$ the second derivatives of $z_i, x_i$ etc. 
The non-degeneracy condition means that the matrix $r$ has rank $n$.

Then the action of $q$ on the jets is given by $x_i=q_i(z)$ and 
$$
R^i_j=\sum_\ell {\partial q_i\over \partial z_\ell} r^\ell_j, \quad
S^i_{jh}=\sum_\ell {\partial q_i\over \partial z_\ell} s^\ell_{jh}+
\sum_m \sum_\ell {\partial q_i\over \partial z_\ell \partial z_m} r^m_h r^\ell_j,\quad {\rm etc}
$$
with summations $\ell, m=1,\ldots, n$.

Thus $U$ is locally parameterized by $t\in\Delta^k$ in the form
$$
(z(t), r(t), s(t),\ldots; q(z(t)), R(t), S(t),\ldots)
$$
where
$$
R^i_j(t)=\sum_\ell {\partial q_i\over \partial z_\ell}  r^\ell_j (t), \, 
S^i_{jh}(t)=\sum_\ell {\partial q_i\over \partial z_\ell} s^\ell_{jh}(t)+
\sum_m \sum_\ell {\partial q_i\over \partial z_\ell \partial z_m} r^m_h(t) r^\ell_j(t),\, {\rm etc}
$$
and the derivatives of $q$ are evaluated at $z(t)$.

We must prove that, as functions of $t$,
$$
\dim {\rm Zcl} (z, r, s, \ldots; x, R, S, \ldots)\ge \dim(z,r,s,\ldots)+\dim G.
$$

We claim that 
$C(z, r, s, \ldots; x, R, S, \ldots)=C(z, r, s, \ldots; x, q_i^{(\nu)}\circ z(t), |\nu|\le \ell)$. 
Clearly the LHS is contained in the right. On the other hand, since $r$ has full rank, 
by our non-degeneracy we may definably find a map $\phi:\Omega\ra \Delta^j$ 
definable in $C(z,r,s,\dots)$ such that $(z,r,s,...)\circ\phi$ is the identity, 
and thus the two fields are equal. Thus we need to prove
$$
\dim {\rm Zcl} (z, r, s, \ldots; x, q_i^{(\nu)}\circ z(t), |\nu|\le \ell) \ge \dim(z,r,s,\ldots)+\dim G.
$$

We consider the differential field containing the functions $z, x, y=q_i^{(\nu)}\circ z$.
The hypotheses of 12.3 hold: that is $(z, x, y)$ is a 
uniformized locus of rank $k={\rm rank}(z)=\dim U$.

If the projection of $U$ to $X$ is not contained in a proper weakly special subvariety
then we have
$$
\dim {\rm Zcl} (z, x, y,\ldots) \ge \dim (z) +\dim G
$$
and the conclusion then follows because, for any sets $A, B$ of functions,
$$
\dim {\rm Zcl} (A, B) -\dim {\rm Zcl} (A) \ge \dim (A,B) - \dim A.
$$
This concludes the proof.\ \qed

\subsection{\bf A special case\/}

We state a special case of Theorem 12.3 which clarifies the 
relationship between
it and the modular and exponential cases. This version views
$q$ as an analogue of $\exp / j$, and concerns differential avatars of
the Cartesian product map
$$
q^\ell: \Omega^\ell\rightarrow X^\ell
$$
which are suitably non-degenerate on each factor so that
the differential equation can be straighforwardly imposed on the
corresponding coordinate functions.

The point is that if $z, x$ have rank $n$ and $(z,x)\in V^k_m$ then
$y$ is uniquely determined and lies in $J^n_r(K)$ by
solving suitable systems of linear equations
in the derivatives of $z, x$.

\medskip
\noindent
{\bf 12.5. Theorem.\/} 
{\it Suppose we have $z=\big(z^{(1)}, \ldots, z^{(\ell)}\big)\in \widehat\Omega^\ell(K)$, 
with ${\rm rank\/} (z^{(k)})=n$ for each $k=1,\ldots, \ell$,
and $x=\big(x^{(1)},\ldots, x^{(\ell)}\big)\in X^\ell(K)$, each of rank $n$, 
such that $x^{(i)}$ is a function of $z^{(i)}$ for each $i$.
Let $y^{(i)}$ be the partial derivatives of $x^{(i)}$ with respect to $z^{(i)}$
up to order $r\ge 2$
and put $y=\big(y^{(1)},\ldots, y^{(\ell)}\big)$.

Suppose further that $(z^{(k)}, x^{(k)})\in V^k_m$ for each $k=1,\ldots, \ell$.
Then
$$
{\rm tr.deg.}_CC\big(z,x,y\big) 
\ge {\rm rank\/} (z)+\ell\dim G
$$
unless $x$ is contained in a proper weakly
special subvariety of $X$. \ \qed\/}

\bigbreak
\bigbreak

{\bf Acknowledgements.\/} For the research reported in this paper,
NM thanks the HKRGC for partial support under grant GRF 17303814, 
JP thanks the EPSRC for partial support under grant EP/N008359/1,
and JT thanks NSERC and the Alfred P. Sloan Foundation for partial support 
under a discovery grant and Sloan Fellowship.

\bigskip

\noindent
\leftline{NM: Department of Mathematics,
University of Hong Kong, Hong Kong.}
\rightline{nmok@hku.hk}

\bigskip
\noindent
\leftline{JP: Mathematical Institute, 
University of Oxford, Oxford, UK.}
\rightline{pila@maths.ox.ac.uk}

\bigskip
\noindent
\leftline{JT: Department of Mathematics, University of Toronto, Toronto, Canada.}
\rightline{jacobt@math.toronto.edu}

\end{document}